\documentclass[11pt]{article}

\usepackage{amsmath, amsthm, amssymb, bbm}
\usepackage{cite, enumerate}
\usepackage{graphicx, color, tikz}

\newtheorem{theorem}{Theorem}[section]
\newtheorem{lemma}[theorem]{Lemma}
\newtheorem{corollary}[theorem]{Corollary}
\newtheorem{proposition}[theorem]{Proposition}

\newtheorem*{definition*}{Definition}

\newcommand{\N}{\mathbb{N}}
\newcommand{\R}{\mathbb{R}}
\newcommand{\E}{\mathbb{E}}

\newcommand{\SP}{\mathbb{S}}
\newcommand{\HY}{\mathbb{H}}
\newcommand{\M}{\mathbb{M}}

\newcommand{\IND}{\mathbbm{1}}
\newcommand{\ST}{\,|\,}
\newcommand{\REARR}{{\mbox{\begin{tiny}$\bigstar$\end{tiny}}}}

\DeclareMathOperator{\CONV}{conv}

\DeclareMathOperator{\INT}{int}
\DeclareMathOperator{\CLOS}{cl}

\DeclareMathOperator{\SPT}{spt}
\DeclareMathOperator{\SN}{sn}
\DeclareMathOperator{\CS}{cs}
\DeclareMathOperator{\SPAN}{span}
\DeclareMathOperator{\ARCOSH}{arcosh}

\newcounter{saveeqn}

\begin{document}
	
\title{Randomized Urysohn--type inequalities}


\author{Thomas Hack\and Peter Pivovarov}


\maketitle

\begin{abstract}
	As a natural analog of Urysohn's inequality in Euclidean
        space, Gao, Hug, and Schneider showed in 2003 that in
        spherical or hyperbolic space, the total measure of totally
        geodesic hypersurfaces meeting a given convex body $K$ is
        minimized when $K$ is a geodesic ball. We present a random
        extension of this result by taking $K$ to be the convex hull
        of finitely many points drawn according to a probability
        distribution and by showing that the minimum is attained for
        uniform distributions on geodesic balls. As a corollary, we
        obtain a randomized Blaschke--Santal{\'o} inequality on the
        sphere.
\end{abstract}

\section{Introduction}

Recent research on isoperimetry in Euclidean space has centered on the
following theme: in the presence of convexity, isoperimetric
principles often admit stronger stochastic versions.  The expository
article \cite{PP:2017b} outlines how work in stochastic geometry has
motivated randomized versions of fundamental inequalities. As an
illustrative example, Urysohn's inequality for convex bodies
$K\subseteq {\mathbb R}^n$, relates volume $V_n$ and mean width $w$,
namely,
 \begin{equation}
  \label{eqn:Urysohn}
  \left(\frac{V_n(K)}{V_n(B)}\right)^{1/n}\leq \frac{w(K)}{w(B)},
 \end{equation}
 where $B$ is the Euclidean unit ball. When viewed as the
 isoperimetric principle that balls minimize mean width under
 prescribed volume, \eqref{eqn:Urysohn} affords a stochastic
 strengthening. Namely, each convex body $K\subseteq {\mathbb R}^n$
 can be approximated from within by a random polytope
 $[K]_N=\mathop{\rm conv}\{X_1,\ldots,X_N\}$, where the $X_i$'s are
 independent random vectors distributed uniformly in $K$. Then one has
 an ``empirical'' version of \eqref{eqn:Urysohn}, i.e.,
 \begin{equation}
\label{eqn:stoch_urysohn}
{\mathbb E} w\left([K]_N\right)\geq {\mathbb E}
w\left(\left[B_K^v\right]_N\right),
\end{equation}where $B_K^v$ is a Euclidean ball with $V_n(K)=V_n(B_K^v)$.
 When $N\rightarrow \infty$, one recovers Urysohn's inequality
 \eqref{eqn:stoch_urysohn} by the law of large numbers. Volumetric
 inequalities for expected mean values have a long history in
 stochastic geometry and go back (at least) to Blaschke's resolution
 of Sylvester's four point problem \cite{Bla:1917}, and its numerous
 generalizations, e.g., arbitrary dimension \cite{Bus:1953},
 \cite{Groe:1974}, \cite{CCG:1999}, compact sets and other intrinsic
 volumes \cite{Pfie:1982}, \cite{HarPao:2003}, continuous
 distributions \cite{PP:2012} (see also \cite[Chapter
   10]{Schneider:2014}).

Alternatively, \eqref{eqn:Urysohn} also means that Euclidean balls
maximize volume for a given mean width. This too admits a stochastic
strengthening, by approximating $K$ from outside by intersections of
halfspaces containing it, say $[K]^N = \bigcap_{i=1}^N H_{i}^{+}$;
here the $H_i^{+}$ are independent random halfspaces with boundary
$H_i$ meeting a neighborhood of $K$, sampled according to the
(suitably normalized) motion invariant measure on such
hyperplanes. Then
\begin{equation}\label{eqn:stoch_urysohn_halfspaces}
{\mathbb E} V_n\left([K]^N\right)\leq {\mathbb E}V_n\left([B_K^w]^N\right),
\end{equation}
where $B_K^w$ is a Euclidean ball with $w(K)=w(B_K^w)$. Again, when
$N\rightarrow \infty$, one obtains \eqref{eqn:Urysohn}. The latter
follows from work in \cite{BorSch:2010}, as explained in \cite[\S
  5]{PP:2017a}. An altogether different approximation of $K$ connected
to \eqref{eqn:Urysohn} by intersections of Euclidean balls with large
radius, as opposed to halfspaces, was studied in \cite{PP:2017a}.

As a natural analogue of \eqref{eqn:Urysohn} on the sphere
$\mathbb{S}^n$, Gao, Hug and Schneider \cite{GHS:2003} considered the
functional $U_1$ defined on the unit sphere by
\begin{align}\label{def-U1-sphere}
U_1(K) = \int_{\mathcal{S}^n_{n-1}} \chi(K\cap S)\, dS,
\end{align}
where $K\subseteq\SP^n$ is a proper spherical convex body,
$\mathcal{S}^n_{n-1}$ the set of $(n-1)$-dimensional great subspheres
in $\SP^n$, equipped with its rotation-invariant measure, and $\chi$
the Euler characteristic.  They showed that
\begin{equation}
\label{eqn:GHS}
U_1(K) \geq U_1(C_K),
\end{equation}
where $C_K$ is a geodesic cap with the same spherical volume as $K$.
The connection to \eqref{eqn:Urysohn} becomes clear in Euclidean space,
when one parametrizes affine hyperplanes via their normal vector and (signed)
distance to the origin: integration over parallel subspaces of the function
$\chi(K\cap \cdot)$ yields precisely the width of $K$ in that direction.

While Urysohn's inequality \eqref{eqn:Urysohn} is just one example
from the rich theory of isoperimetry for intrinsic volumes e.g.,
\cite{Schneider:2014}, the situation on the sphere is wide open.  In
\cite{GHS:2003}, three natural families of functionals that could be
considered spherical intrinsic volumes are discussed and most
isoperimetric inequalities are still conjectural. However, progress
has been made in hyperbolic space \cite{WanXia:2014},
\cite{AndWei:2018}. Beyond intrinsic volumes, there is much recent
interest in extending fundamental notions from Euclidean space to the
sphere and beyond, e.g., floating bodies \cite{BW:2018}; polytopal
approximation \cite{BLW:2018}; behavior of asymptotic mean values of
functionals of random polytopes \cite{BHRS:2017}.

The randomized forms of Urysohn's inequality \eqref{eqn:stoch_urysohn}
and \eqref{eqn:stoch_urysohn_halfspaces} are just two examples of
stochastic inequalities in $\mathbb{R}^n$. Other fundamental
inequalities like those of Brunn--Minkowski \cite{Gardner:BM} and
Blaschke--Santal\'{o} \cite{Sant:1949} also admit stochastic
strengthenings \cite{PP:2017b}, \cite{Cordero:etal:2015}. Moreover,
centroid bodies and their $L_p$ \cite{LYZ2000a} and Orlicz extensions
\cite{LYZ:2010} also admit stochastic forms \cite{PP:2012},
\cite{PP:2017b}. Interest is driven in part by applications to
high-dimensional probability, especially small-ball probabilities
\cite{PP:2013}, \cite{PP:2017b}.  In comparison, stochastic
isoperimetry on the sphere is in early stages.  Recently, spherical
centroid bodies and their empirical analogues were introduced and
stochastic isoperimetric inequalities were established in
\cite{BHPS:2019}.

In this paper we address an empirical version of the spherical Urysohn
inequality \eqref{eqn:GHS}. In fact, as in \cite{GHS:2003}, our
treatment will also include hyperbolic space. Throughout, we let
$\M^n$ be either spherical, Euclidean, or hyperbolic space equipped
with its isometry-invariant volume measure $\lambda_n$, and consider
the natural analog of \eqref{def-U1-sphere} in $\M^n$ (see Section 2
for precise definitions). Our first main theorem then reads as
follows:

\begin{theorem}\label{1:main-thm}
Let $N\in\N$, $f_1,\ldots, f_N\colon\M^n\to\R^+$ bounded,
integrable, and with proper support, if $\M^n = \SP^n$. Set
\[
  I(f_1,\ldots, f_N) = \int_{\M^n}\ldots\int_{\M^n}
  U_1(\CONV\{x_1,\ldots, x_N\}) \prod_{i=1}^N f_i(x_i)\, dx_1\ldots
  dx_N.
\]
Then
\begin{align}\label{eq:main-thm-U1}
  I(f_1\ldots, f_N) \geq I(\|f_1\|_\infty B_1, \ldots, \|f_N\|_\infty B_N),
\end{align}
where the $B_i$ are geodesic balls in $\M^n$ centered at a common point, 
satisfying \[\lambda_n(B_i) = \frac{\|f_i\|_{L^1(\M^n)}}{\|f_i\|_\infty}. \]
\end{theorem}

By taking indicator functions on compact sets, we obtain the following
corollary.

\begin{corollary}
  \label{1:main-cor}
	Let $N\in\N$, $K\subseteq\M^n$ be compact, and proper in the case 
	$\M^n =\SP^n$. Then
	\[
		\E U_1([K]_N) \geq \E U_1([B_K]_N),
	\]
	where $B_K$ is a geodesic ball satisfying $\lambda_n(K) =
        \lambda_n(B_K)$.
\end{corollary}

The proof of Theorem \ref{1:main-thm} does not rely on
\eqref{eqn:GHS}. Rather, we first prove a rearrangement inequality for
$I(f_1,\ldots,f_N)$ that reduces the problem to radially decreasing
densities. This is similar to the route taken in the Euclidean setting
\cite{PP:2012}, \cite{PP:2017b} but we use two-point symmetrization,
e.g., \cite{Wolontis:1952}, \cite{Benyamini:1984}, \cite{BroSol:2000},
as in \cite{GHS:2003}, rather than Steiner symmetrization. In
${\mathbb R}^n$, a key tool behind stochastic isoperimetric
inequalities is Steiner symmetrization and associated rearrangement
inequalities \cite{Rogers:1957}, \cite{BLL:1974}, \cite{Christ:1984}.
The latter interface well with fundamental tools in convex geometry,
like shadow systems, Alexadrov's theory of mixed volumes, Busemann's
convexity of intersection bodies (see \cite{PP:2017b} and the
references therein).  There is no comparable development on the sphere
or in hyperbolic space.  Going back to Baernstein and Taylor 
\cite{BaerTay:1976}, two-point symmetrization has been used as an
analytical tool on $\mathbb{M}^n$ for multiple integral
rearrangement inequalities, see also \cite{BurSch:2001},
\cite{Morp:2002}, \cite{BurHaj:2006}; more recently in isoperimetric inequalities
\cite{Bezdek:2018}. However, such techniques have not yet been fused
with stochastic convex geometry in $\M^n$.  Theorem \ref{1:main-thm}
is a first step in this direction.

As noted in \cite{GHS:2003}, there is a special relationship between
$U_1$ and spherical polar duality (see Proposition \ref{U1polar}). In
this way, \eqref{eq:main-thm-U1} can also be reinterpreted as a
spherical Blaschke-Santal\'{o} inequality in stochastic form.

\begin{theorem}
  \label{1:main-thmPolar}
	Let $N\in\N$, $f_1,\ldots, f_N\colon\SP^n\to\R^+$ bounded,
	integrable, and with proper support. Set
  \[
  \tilde{I}(f_1,\ldots, f_N) = \int_{\SP^n}\ldots\int_{\SP^n}
  \lambda_n(\CONV\{x_1,\ldots, x_N\}^*) \prod_{i=1}^N f_i(x_i)\,
  dx_1\ldots dx_N.
  \]
  Then
  \[
  \tilde{I}(f_1\ldots, f_N) \leq \tilde{I}(\|f_1\|_\infty C_1, \ldots,
  \|f_N\|_\infty C_N),
  \]
  where the $C_i$ are spherical caps centered at a common
  point, satisfying \[\lambda_n(C_i) =
  \frac{\|f_i\|_{L^1(\SP^n)}}{\|f_i\|_\infty}. \]
\end{theorem}

Again, by taking indicator functions on compact sets, we obtain as a corollary:
\begin{corollary}\label{1:cor:randBS}
  Let $N\in\N$, $K\subseteq\SP^n$ be compact, and proper. Then
  \[
  \E \lambda_n([K]_N^*) \leq \E \lambda_n([C_K]_N^*),
  \]
  where $C_K$ is a spherical cap satisfying $\lambda_n(K) = \lambda_n(C_K)$.
\end{corollary}

A symmetric version of Corollary \ref{1:cor:randBS}, where $[K]_N$ is
replaced by the convex hull of the random points $X_i$ and their
reflections about some fixed origin, can be obtained by using a stochastic
Blaschke--Santal{\'o} inequality in Euclidean space from \cite{Cordero:etal:2015},
together with gnomonic projection. This strategy of relying on Euclidean
techniques such as Steiner symmetrization in $\mathbb{R}^n$ after
projecting from the sphere has also been studied in detail for
spherical centroid bodies in \cite{BHPS:2019}.

By using two-point symmetrization, we treat Euclidean, spherical and
hyperbolic space simultaneously.  Note that in addition to Steiner
symmetrization, which is used to prove \eqref{eqn:stoch_urysohn}, and
Minkowski symmetrization for \eqref{eqn:stoch_urysohn_halfspaces},
this gives another method to obtain a randomized Urysohn inequality in
$\R^n$.

\section{Background material}\label{sec:background}

In this section, we introduce our models for the spaces of constant
curvature, along with certain isometry-invariant measures associated
with them. Moreover, we give a short account of two-point
symmetrization and rearrangements in these spaces and collect
properties that we need later on.  As a general reference for
spherical convexity in particular, we refer the reader to
\cite{Glas:1996} or \cite{SW:08}.  Further information about
two-point symmetrization can also be found in \cite{AubFra:2004},
\cite{Baer:2019}, \cite{BaerTay:1976}, \cite{BroSol:2000}, and \cite{BurFor:2013}.
 
\subsection{Spaces of constant curvature}

We will use the following models $\M^n$ for spherical, Euclidean, and hyperbolic
geometry: Let $e := (1, 0, \ldots, 0)^T\in\R^{n+1}$ and write 
$x\cdot y := x_0y_0 + x_1y_1 + \ldots x_ny_n$ and 
$\langle x, y\rangle := x_0y_0 - x_1y_1 - \ldots x_ny_n$ for the standard
Euclidean and Minkowski scalar products. Then we define
\begin{itemize}
	\item Spherical space: 
          $\SP^n := \{x\in\R^{n+1}\colon x \cdot x = 1\}$,
	\item Euclidean space: 
	       $\R^n := \{x\in\R^{n+1} \colon x\cdot e = 1\}$,
	\item Hyperbolic space:
	       $\HY^n := \{x\in\R^{n+1} \colon x\cdot e > 0,
	               \langle x, x \rangle = 1\}$
\end{itemize}
We think of $\R^n$ as an affine subspace, but we will still refer to $\SP^{n-1}$
as $\SP^n \cap \{x_{n+1} = 0\}$. 
Topological interior and closure will always be understood relative to $\M^n$ 
and written as $\INT$ and $\CLOS$, respectively. The isometry-invariant volume 
measure will be denoted by $\lambda_n$ simultaneously for all three geometries.
In each case it is the restriction of $n$-dimensional Hausdorff measure on 
$\R^{n+1}$. We will just write $dx$ instead of $d\lambda_n(x)$ when integrating 
over $\M^n$.

Similarly, we will denote by $\mu_{n-1}^n$ the motion-invariant measure on
$\mathcal{M}_{n-1}^n$, the collection of $(n-1)$-dimensional totally geodesic
submanifolds of $\M^n$, that is,
\[
  \mathcal{M}_{n-1}^n = \{ E \cap \M^n \ST E \text{ is an } 
                        n\text{-dimensional subspace of } \R^{n+1}\}.
\]
The normalization is $\mu_{n-1}^n(\mathcal{M}_{n-1}^n) = 1$, if $\M^n = \SP^n$, 
and such that 
$\mu_{n-1}^n(\{M\in\mathcal{M}_{n-1}^n\ST M\cap B_1\neq \emptyset\}) %
= 1$, if $\M^n = \R^n$ or $\HY^n$, where $B_1$ is the 
geodesic ball of radius $1$ around $e$.
For every $M\in\mathcal{M}_{n-1}^n$, we can find a vector $u\in\R^{n+1}$, such
that $M = u^\bot\cap\M^n$, where the orthogonal complement is taken either with
respect to Euclidean or Minkowski scalar product in $\R^{n+1}$.
Again, we write $dM$ instead of $d\mu_{n-1}^n(M)$ when integrating over
$\mathcal{M}_{n-1}^n$. 

Let $K\subseteq\M^n$. If for any $x, y\in K$, $x\neq -y$ in the case 
$\M^n = \SP^n$, the shortest geodesic segment $[x, y]$ connecting $x$, $y$
lies inside $K$, we call $K$ \emph{convex}. Moreover, $K$ is a 
\emph{convex body} if it is convex and compact. The \emph{convex hull} of a set 
$K$, denoted by $\CONV K$, is the intersection of all convex sets containing 
$K$. A closed set $K\subseteq\SP^n$ is called \emph{proper}, if it is
contained in an open hemisphere. Accordingly, $f\colon\SP^n\to\R$ is said
to have \emph{proper support}, if $\SPT f = \CLOS\{x\in\SP^n\ST f(x) \neq 0\}$
is proper.

For a set $A\subseteq\M^n$, we define
\begin{align}\label{2:def-chi}
  \chi(A) := \begin{cases}
	  1,	& \text{if } A\neq\emptyset, \\
	  0,	& \text{if } A = \emptyset.
  \end{cases}
\end{align}
Note that on convex bodies in $\R^n$ and $\HY^n$, and on \emph{proper} 
convex bodies in $\SP^n$, $\chi$ equals the Euler characteristic. In some 
parts we will also allow densities having non-proper support in $\SP^n$, 
hence, it seems more appropriate to use the above version of $\chi$. 
We can now give a definition of $U_1$ on $\M^n$ analogous to 
\eqref{def-U1-sphere}:
\begin{align*}
	U_1(K) := \int_{\mathcal{M}^n_{n-1}} \chi(K\cap M)\, dM,
\end{align*}
where $K\subseteq\M^n$ is compact and $\chi$ as in \eqref{2:def-chi}.

\medskip
\noindent{\it Polar coordinates.} We set
\[
	R^\M := \begin{cases}
		\pi,  		& \text{if } \M^n = \SP^n, \\
		\infty, 	& \text{if } \M^n = \R^n \text{ or } \M^n = \HY^n.
	\end{cases}
\]
For $t\in \left[0, R^\M\right]$, we define
\[
  \CS t := \begin{cases}
  	\cos t,  	& \text{if } \M^n = \SP^n, \\
  	1, 				& \text{if } \M^n = \R^n,  \\
  	\cosh t,	& \text{if } \M^n = \HY^n.
  \end{cases}
  \qquad
  \SN t := \begin{cases}
  	\sin t,  	& \text{if } \M^n = \SP^n, \\
  	t, 				& \text{if } \M^n = \R^n,  \\
  	\sinh t,	& \text{if } \M^n = \HY^n,
  \end{cases}
\]
and introduce polar coordinates $x(t, u) := e\CS t + u\SN t$, for 
$u\in\SP^{n-1}$ and $t\in\left[0, R^\M\right]$. The following transformation
formula holds in all three geometries:
\begin{align}\label{eq:polar}
  \int_{\M^n} f(x) dx = \int_{\SP^{n-1}} \int_0^{R^\M}
                        f(x(t, u))\, \SN^{n-1}t\, dt du,
\end{align}
for any integrable function $f\colon\M^n\to\R$. For a proof see 
\cite[Sections 3.F, 3.H]{Gallot:etal:2004}.

\medskip
\noindent{\it Metric.} We write $d_{\M^n}(x, y)$ for the geodesic distance
between $x, y\in\M^n$, that is,
\[
d_{\M^n}(x, y) = \begin{cases}
\arccos(x\cdot y),  	  & \text{if } \M^n = \SP^n, \\
\|x-y\|,  			  & \text{if } \M^n = \R^n,  \\
\ARCOSH\left(\langle x, y\rangle\right),
& \text{if } \M^n = \HY^n.
\end{cases}
\]
In polar coordinates, we have $d_{\M^n}(e, x(t, u)) = t$, for 
$u\in\SP^{n-1}$ and $t\in\left[0, R^\M\right]$.

\medskip
\noindent{\it Polarity.} The following relation is special to the case $\M^n = \SP^n$: For 
$K\subseteq\SP^n$, its \emph{polar set} $K^*$ is given by
\[
  K^* = \{x\in\SP^n\ST x\cdot y \leq 0\textrm{ for all } y\in K \}.
\] 
\begin{proposition}\label{U1polar}
Let $K\subseteq\SP^n$ be a convex body. Then
\[
	\frac{U_1(K)}{\mu^n_{n-1}(\mathcal{M}^n_{n-1})}  
	+ \frac{2\lambda_n(K^*)}{\lambda_n(\SP^n)} = 1.
\]
\end{proposition}
\begin{proof}
If $K$ is	proper the proof can be found in \cite[eq. 20]{GHS:2003}.
On the other hand, as soon as $K$ contains antipodal points, we have
$U_1(K) = \mu^n_{n-1}(\mathcal{M}^n_{n-1})$ and $\INT K^* = \emptyset$.
\end{proof}

\subsection{Two-point symmetrization and rearrangements}

A hyperplane $H\in\mathcal{M}_{n-1}^n$ divides $\M^n$ into two connected
components, which we will call the closed halfspaces $H^+$ and $H^-$ in 
such a way that always $e\in H^+$. The associated orthogonal reflection
about $H$ will be denoted by $\rho\colon \M^n\to \M^n$. 
If $H = u^\bot\cap\M^n$, for $u\in\R^{n+1}$, then $\rho$ is given by
\[
\rho x := \rho(x) = \begin{cases}
	x - 2\frac{x\cdot u}{u\cdot u}u,
										  	  & \text{if } \M^n = \SP^n, \\
	x - 2\frac{x\cdot u}{1-(u\cdot e)^2}[u - (u\cdot e)e],
								  			  & \text{if } \M^n = \R^n,  \\
	x - 2\frac{\langle x, u\rangle}{\langle u, u\rangle}u,
													& \text{if } \M^n = \HY^n.
\end{cases}
\]
In view of the following, we make a note of the fact that for any $x,y\in H^+$ 
we have
\[
  d_{\M^n}(x,y) \leq d_{\M^n}(x,\rho(y)).
\]
Indeed, let $z$ be the intersection of the geodesic segment $[x, \rho y]$ with $H$.
Then 
\[
	d_{\M^n}(x,y) \leq d_{\M^n}(x,z) + d_{\M^n}(z,y) 
	=  d_{\M^n}(x,z) +  d_{\M^n}(z,\rho y) 
	=  d_{\M^n}(x,\rho y),
\] 
by the triangle inequality.

\medskip
\noindent{\it Two-point symmetrization.} If we decompose a set $K\subseteq\M^n$ as
\[
  K = \underbrace{(K \cap \rho K)}_{K^\textrm{sym}} \dot{\cup} 
  			\underbrace{\left[(K \cap H^+)
  				\setminus K^\textrm{sym}\right]}_{K^\textrm{fix}} \dot{\cup} 
  			\underbrace{\left[(K \cap H^-)
  				\setminus K^\textrm{sym}\right]}_{K^\textrm{mov}},
\]
the \emph{two-point symmetrization} $T = (H, \rho, T)$ with respect to $H$ is
given by
\[
  TK = \underbrace{(K \cap \rho K)}_{K^\textrm{sym}} \dot{\cup} 
  			\underbrace{\left[(K \cap H^+)
  				\setminus K^\textrm{sym}\right]}_{K^\textrm{fix}} \dot{\cup} 
  			\underbrace{\rho\left[(K \cap H^-)
  				\setminus K^\textrm{sym}\right]}_{\rho K^\textrm{mov}}.
\]
\begin{figure}
	\begin{center}
	\resizebox{90mm}{!} {
		\begin{tikzpicture}[scale=1.6, xscale=1.2]
			\definecolor{darkgreen}{RGB}{50,110,0}
			\definecolor{darkblue}{RGB}{0,0,180}


			\draw[dashed] (-2, 0) -- (2, 0) node[right]{$H$};	\node[above] at (2, 0.8) {$H^+$};
			\node[below] at (2, -0.8) {$H^-$};

			\node at (-1.6, 1.1) [circle, fill, color=gray, inner sep=1pt]{};
			\node at (-1.6, 1.1) [above left, color=gray]{$e$};

			\newcommand\Body{
				plot [smooth cycle, tension=1, rotate=40, scale=1.5]
				coordinates {(1, 0) (0.2, 0.5) (-1, 0) (0.2, -0.5)}
			}

			\draw[dotted] \Body;

			\draw[yscale=-1, dotted] \Body;

			\node[below, color=gray] at (-0.95, -1.1) {$K$};
			\node[above, color=gray] at (-0.9, 1.05) {$\rho K$};

			\node[color=black] at (0.05, 0.2) {$K^\textrm{sym}$};
			\node[color=black] at (0.8, 0.8) {$K^\textrm{fix}$};
			\node[color=gray] at (-0.85, -0.7) {$K^\textrm{mov}$};


			\begin{scope}	
				\clip \Body;
				\fill[yscale=-1, color=blue, opacity=0.1] \Body;
			\end{scope}

			\begin{scope}center
				\clip (-2, 0) rectangle (2, 2);
				\clip \Body;
				\fill[even odd rule, color=blue, opacity=0.1] 
				{\Body} {[yscale=-1]\Body};
			\end{scope}

			\begin{scope}
				\clip (-2, 0) rectangle (2, 2);
				\clip {[yscale=-1]\Body};
				\fill[even odd rule, color=blue, opacity=0.1] 
				{\Body} {[yscale=-1]\Body};
			\end{scope}
			\node[color=black] at (-0.8, 0.7) {$\rho K^\textrm{mov}$};
			\node[color=darkblue] at (-0.0, -0.45) {$TK$};
		\end{tikzpicture}
	}
	\end{center}
	\caption{A two-point symmetrization of $K$.}
\end{figure}
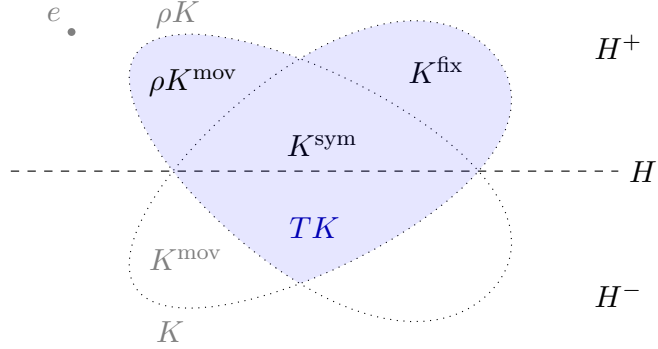
Note that all unions are disjoint up to sets of measure zero, which
immediately shows that $\lambda_n(TK) = \lambda_n(K)$ for all
measurable sets $K\subseteq\M^n$. Intuitively, $T$ pushes as much mass
as possible towards $e$ (that is, into $H^+$) without double-covering
points.

Two-point symmetrization of a function $f\colon\M^n\to\R^+$ is given by
\[
  Tf(x) = \begin{cases}
    \max\{f(x), f(\rho x)\}, & \text{if } x\in H^+, \\
    \min\{f(x), f(\rho x)\}, & \text{if } x\in H^-.
  \end{cases}
\]
We have $T\IND_K = \IND_{TK}$ for any set $K\subseteq\M^n$. More generally, one can 
check that
\begin{equation}\label{2:eqTPlevel}
  \{Tf>s\} = T\{f>s\}
\end{equation}
for all $s>0$.
For a continuous function $f\in C(\M^n)$ and $\delta>0$, denote by
\[
  \omega(\delta, f) = \sup\{|f(x)-f(y)|\colon d_{\M^n}(x,y)<\delta,\, x,y\in\M^n\}
\]
the \emph{modulus of continuity} of $f$. The proof of the next lemma can be found, e.g.,
in \cite[Proposition 1.37]{Baer:2019} for $\M^n = \R^n$ and translated verbatim to $\SP^n$
and $\HY^n$ (cf. \cite[Section 7.1]{Baer:2019}).

\begin{lemma}\label{2:lemModCont}
For every continuous function $f\in C(\M^n)$, $\delta>0$, and every two-point
symmetrization $T$, we have
\[
  \omega(\delta, Tf) \leq \omega(\delta, f).
\]
\end{lemma}

\medskip
\noindent{\it Rearrangements.} Let $f\colon\M^n\to\R^+$ be integrable so that 
in particular $\lambda_n(\{f>t\})<\infty$ for all $t>0$. We want to associate to $f$
a function $f^*$ whose sublevel sets are geodesic balls around $e\in\M^n$. 
The \emph{symmetric decreasing rearrangement} of $f$ is the function 
$f^\REARR\colon\M^n\to\R^+$ of $x = x(t, u)$, depending only on $t$, is 
non-increasing as $t$ increases, and has the property 
\[
  \lambda_n(\{f>s\}) = \lambda_n(\{f^\REARR>s\})
\]
for all $s>0$. It is defined up to sets of measure zero and can be written
explicitly as
\[
  f^\REARR(x(t, u)) = \inf\{s\ST \lambda_n(\{f > s\}) \leq 
                			\lambda_n(B_t)\},
\]
where $B_t$ is the geodesic ball around $e$ with radius
$t = d_{\M^n}(e,x)$.

The next lemma relates two-point symmetrization and symmetric decreasing 
rearrangements.

\begin{lemma}\label{2:lemTPrearr}
Let $f\colon\M^n\to\R^+$ be integrable and  $T = (H, \rho, T)$ any 
two-point symmetrization. Then the following holds:
\begin{enumerate}[(i)]
  \item For $x\in H^+$ we have $f(x)f^\REARR(x) + f(\rho x)f^\REARR(\rho x) \leq 
	    Tf(x)f^\REARR(x) + Tf(\rho x)f^\REARR(\rho x)$.
  \item $\int_{\M^n} |f(x) - f^\REARR(x)|^2\, dx \geq 
        \int_{\M^n} |Tf(x) - f^\REARR(x)|^2\, dx$.
\end{enumerate}
\end{lemma}
\begin{proof}
For {\it (i)}, we refer to \cite[Theorem 2.8 (a)]{Baer:2019}, whereas {\it (ii)} follows from integration over $H^+$ and the fact that $T$ and $^\REARR$ preserve $L^2$-norms (cf. \cite[eq.~(2.13)]{Baer:2019}).
\end{proof}


\section{Auxiliary results}

In this section, we compile results from the literature that are
necessary to complete the proof of our main result. We consider
functionals of the form
\begin{align}\label{3:defI}
  I_\Psi(f_1,\ldots, f_N) = \int_{\M^n}\ldots\int_{\M^n} \Psi(x_1,\ldots, x_N)
                       	    \prod_{i=1}^N f_i(x_i)\, dx_1\ldots dx_N,
\end{align}
for a bounded, measurable function $\Psi\colon(\M^n)^N\to\R^+$ and bounded, 
integrable functions $f_i\colon\M^n\to\R^+$, $1\leq i\leq N$. Sometimes, we 
will also consider the truncated functional
\[
  I^R_\Psi(f_1,\ldots, f_N) = I_\Psi(\IND_{B_R}f_1,\ldots, \IND_{B_R}f_N),
\]
where $B_R$ is the geodesic ball of radius $R>0$ around the origin $e\in\M^n$. 
Clearly, in the case $\M^n = \SP^n$, we have $I^R = I$ whenever $R\geq\pi$.
Integrals of type \eqref{3:defI}, involving monotone functions of pairwise distances $d_{\M^n}(x_i, x_j)$, have already been considered by Morpurgo \cite{Morp:2002}, Burchard and Schmuckenschl\"{a}ger \cite{BurSch:2001}, and Burchard and Hajaiej \cite{BurHaj:2006}. Before treating the functional $\Psi(x_1, \ldots, x_N) = U_1(\CONV\{x_1, \ldots, x_N\})$, we recall a general convergence result for iterated two-point symmetrizations.

\subsection{Achieving radial symmetry}

The following proposition is due to Baernstein and Taylor \cite{BaerTay:1976} 
in the case $\M^n = \SP^n$ (see also \cite[Section 2.5]{Baer:2019}). We 
reproduce their proof at once in all three geometries.

\begin{proposition}\label{3:propTPRadial}
Let $f_1,\ldots, f_N\colon\M^n\to\R^+$ be bounded, integrable functions, and 
let $\Psi(x_1,\ldots, x_N)\colon(\M^n)^N\to\R^+$ be bounded and measurable. 
Furthermore, assume that
\[
  I_\Psi(f_1,\ldots, f_N) \geq I_\Psi(Tf_1,\ldots, Tf_N),
\]
for every two-point symmetrization $T$ in $\M^n$. Then
\[
  I_\Psi(f_1,\ldots, f_N) \geq I_\Psi(f_1^\REARR,\ldots, f_N^\REARR).
\]
\end{proposition}
\begin{proof}
We follow \cite[Section 2]{BaerTay:1976}. We start with the following facts:
\begin{enumerate}[(i)]
	\item $\IND_{B_R}f_i \to f_i$ in $L^1(\M^n)$ as $R\to\infty$,
	\item $|I_\Psi(f_1,\ldots, f_N)| \leq \|\Psi\|_\infty \prod_{i=1}^N \|f_i\|_{L^1(\M^n)}$,
	\item the map $f\mapsto f^\REARR$ is continuous in $L^1(\M^n)$
	\item there are sequences $(\phi_i^j)_{j\in\N}$ in $C_c(\M^n)$, $\SPT\phi^j_i\subseteq \SPT f_i$, such that $\phi^j_i\to f_i$ in $L^1(\M^n)$
\end{enumerate}
Indeed, (i) follows from dominated convergence, (ii) is a standard estimate, (iii) follows
from the non-expansivity of symmetric decreasing rearrangements 
(see \cite[Theorem 3.5]{LL:Analysis}), and (iv) is an application of Urysohn's lemma.

By (i) and (ii), we can assume without loss of generality that 
$\SPT f_i\subseteq B_R$, and by (ii), (iii) and (iv), we can restrict further 
to continuous functions $f_i$ (supported in $B_R$) and denote this space by 
$C(B_R)\subseteq L^2(B_R)$. We define for $f\in C(B_R)$
\[
  \mathcal{S}(f) := \{F\in C(B_R)\colon \omega(\cdot, F)\leq\omega(\cdot, f)
                   \;\textrm{and}\; F^\REARR=f^\REARR\}.
\]
We show now that $\mathcal{S}(f)$ is a compact subset of $L^2(B_R)$. Since $f$ 
is uniformly continuous, for each $\varepsilon>0$ there exists $\delta>0$, such 
that $\omega(\delta, F) \leq \omega(\delta, f) < \varepsilon$, for all 
$F\in\mathcal{S}(f)$. Moreover, $\|F\|_\infty = \|F^\REARR\|_\infty = 
\|f^\REARR\|_\infty$. Hence, $\mathcal{S}(f)$ is a uniformly equicontinuous, 
uniformly bounded family of continuous functions, thus relatively compact in 
$(C(B_R), \|\cdot\|_\infty)$ by the Arzel\'{a}-Ascoli theorem. Since the map
$f\mapsto f^\REARR$ is continuous also in the $(C(B_R), \|\cdot\|_\infty)$ topology 
(take $p\to\infty$ in \cite[Theorem 3.5]{LL:Analysis}) and the sets 
$\{\omega(\delta, F)\leq\omega(\delta, f)\}$ and therefore
\[
  \{\omega(\cdot, F)\leq\omega(\cdot, f)\} =
  \bigcap_{\delta\in\R^+}
	\left\{\omega(\delta, F)\leq\omega(\delta, f)\right\},
\]
are closed (by the triangle inequality), the set $\mathcal{S}(f)$ is compact in 
$(C(B_R), \|\cdot\|_\infty)$. It is also compact in 
$(C(B_R), \|\cdot\|_{L^2(B_R)})$, since by 
$\|f\|_{L^2(B_R)}^2 \leq \lambda_n(B_R)\|f\|_\infty^2$, the latter space 
has a coarser topology.

Next, by Lemma \ref{2:lemModCont}, $T\mathcal{S}(f) \subseteq \mathcal{S}(f)$ 
for every two-point symmetrization $T$. Since, by the Cauchy--Schwarz inequality, 
we have
\[
  |I_\Psi(f_1,\ldots, f_N)| \leq \lambda_n(B_R)^N \|\Psi\|_\infty 
                            \prod_{i=1}^N \|f_i\|_{L^2(B_R)},
\]
for fixed $f_1,\ldots, f_N\in C(B_R)$, the set
\[
  \mathcal{P} := \{(F_1,\ldots, F_N)\in\mathcal{S}(f_1)\times\cdots\times
                 \mathcal{S}(f_N)\colon I(f_1,\ldots, f_N)\geq I(F_1,\ldots,F_N)\}
\]
is compact in $L^2(B_R)\times\cdots\times L^2(B_R)$. By assumption, 
$\mathcal{P}$ is closed under simultaneous two-point symmetrizations 
$(F_1,\ldots, F_N)\mapsto(TF_1,\ldots, TF_N)$. We are done, if we can show 
that $\mathcal{P}$ contains $(f_1^\REARR,\ldots, f_N^\REARR)$.

By compactness, there exist $(F_1^0,\ldots, F_N^0)\in\mathcal{P}$ such that
\[
  \sum_{i=1}^N\|F_i^0-f_i^\REARR\|_{L^2(B_R)} = 
    \min\left\{\sum_{i=1}^N\|F_i-f_i^\REARR\|_{L^2(B_R)}\colon
          (F_1,\ldots, F_N)\in\mathcal{P}\right\}.
\]
Without loss of generality, assume that $F_1^0 \neq f_1^\REARR$, that is, there 
exists $t>0$ such that
\[
  E_1 := \{F_1^0 > t\} \neq \{f_1^\REARR > t\} =: E_2, \quad E_1, E_2\subseteq B_R.
\]
Since $\lambda_n(E_1)=\lambda_n(E_2)$, there exist 
$x\in\INT (E_1\setminus E_2)$ and $y\in\INT (E_2\setminus E_1)$. Let 
$H\in\mathcal{M}^n_{n-1}$ be the totally geodesic hypersurface that 
orthogonally bisects the geodesic segment $[x, y]$, $\rho$ the reflection 
about $H$, and $T$ the associated two-point symmetrization. Then we can find 
a small geodesic ball $B$ around $y$, such that 
$B\subseteq (E_2\setminus E_1)$ and $\rho B \subseteq (E_1 \setminus E_2)$.

For all $z\in B$, we then have $f_1^\REARR(z)>t\geq f_1^\REARR(\rho z)$ and 
$F_1^0(\rho z)> t\geq F_1^0(z)$, yielding
\[
  F_1^0(z)f_1^\REARR(z) + F_1^0(\rho z)f_1^\REARR(\rho z) < 
  TF_1^0(z)f_1^\REARR(z) + TF_1^0(\rho z)f_1^\REARR(\rho z).
\]
Since, by Lemma \ref{2:lemTPrearr} (a), the same inequality holds with ``$\leq$'' for all 
$z\in B_R$, we get
\[
  \int_{B_R} F_1^0(z)f_1^\REARR(z)\, dz < \int_{B_R} TF_1^0(z)f_1^\REARR(z)\, dz,
\]
and, since $T$ preserves $L^2$-norms,
\begin{align*}
  \int_{B_R} |F_1^0 - f_1^\REARR|^2 
  &= \int_{B_R} (F_1^0)^2 - 2F_1^0f_1^\REARR + (f_1^\REARR)^2 \\
  &> \int_{B_R} (TF_1^0)^2 - 2TF_1^0f_1^\REARR + (f_1^\REARR)^2
  = \int_{B_R} |TF_1^0 - f_1^\REARR|^2.
\end{align*}
Moreover, by \ref{2:lemTPrearr} (b) we have
\[
  \int_{B_R} |F_i^0 - f_i^\REARR|^2 \geq \int_{B_R} |TF_i^0 - f_i^\REARR|^2 
\]
for $2\leq i\leq N$, hence,
\[
  \sum_{i=1}^N\|F_i^0-f_i^\REARR\|_{L^2(B_R)} >
  \sum_{i=1}^N\|TF_i^0-f_i^\REARR\|_{L^2(B_R)}.
\]
which, since $(TF_1^0,\ldots, TF_N^0)\in\mathcal{P}$, is a contradiction.
\end{proof}

\subsection{From rotation invariance to balls}

We explain now how one can progress further from bounded radially
symmetric densities to indicator functions of geodesic
balls. In doing so, we make use of polar coordinates in $\M^n$. First,
we need the following version of \cite[Lemma 3.5]{PP:2012}:

\begin{lemma}\label{lem:bathtub}
Let $f\colon\left[0, R^\M\right]\to\R^+$ be bounded, measurable and assume that
\[
  \int_0^{R^\M} f(t)\SN^{n-1}t\, dt < \infty.
\]
Define $A\in\left[0, R^\M\right]$ such that
\[
  \int_0^{R^\M} f(t)\SN^{n-1}t\, dt = \int_0^A \|f\|_\infty \SN^{n-1}t\, dt.
\]
Then for any increasing function $\phi\colon\left[0, R^\M\right]\to\R^+$,
\[
  \int_0^{R^\M} \phi(t)f(t)\SN^{n-1}t\, dt \geq
  \int_0^{R^\M} \phi(t)\|f\|_\infty\mathbbm{1}_{[0,A]}(t)\SN^{n-1}t\, dt.
\]
\end{lemma}
\begin{proof}
By the monotonicity of $\phi$ and the positivity of $\SN$ on
$\left[0, R^\M\right]$, we can estimate
\begin{align*}
  \int_0^{R^\M} \phi(t)f(t)\SN^{n-1}t\, dt &= 
		   \int_0^A \phi(t)f(t)\SN^{n-1}t\, dt + \int_A^{R^\M}
		   \phi(t)f(t)\SN^{n-1}t\, dt \\
		&  \geq \int_0^A \phi(t)f(t)\SN^{n-1}t\, dt +
		   \phi(A)\int_A^{R^\M} f(t)\SN^{n-1}t\, dt \\
		&= \int_0^A \phi(t)f(t)\SN^{n-1}t\, dt + 
		   \phi(A)\int_0^A \left(\|f\|_\infty - f(t)\right)
		   \SN^{n-1}t\, dt \\
		&  \geq \int_0^A \phi(t)f(t)\SN^{n-1}t\, dt + 
		   \int_0^A \phi(t)\left(\|f\|_\infty - f(t)\right)\SN^{n-1}t\, dt \\
		&= \int_0^A \phi(t)\|f\|_\infty \SN^{n-1}t\, dt,
\end{align*}
which gives the statement.
\end{proof}

Changing to polar coordinates
$x(t, u) = e\CS t + u \SN t$, $t\in\left[0, R^\M\right]$, $u\in\SP^{n-1}$
(see Section \ref{sec:background}), we can formulate the next proposition.
It is a variant of \cite[Proposition 3.9]{PP:2012}.

\begin{proposition}\label{3:propRadCap}
Let $f_1,\ldots, f_N\colon\M^n\to\R^+$ be bounded, integrable functions, and
let $\Psi(x_1,\ldots, x_N)\colon(\M^n)^N\to\R^+$ be bounded, measurable, such
that the function
\[
  \phi(t_1,\ldots, t_N) = \int_{\SP^{n-1}}\ldots\int_{\SP^{n-1}}
						  \Psi(x(t_1, u_1), \ldots, x(t_N, u_N))\,
						  du_1\ldots du_N
\]
is increasing in each coordinate. Then
\[
  I_\Psi(f_1^\REARR,\ldots, f_N^\REARR) 
  \geq I_\Psi(\|f_1\|_\infty B_1,\ldots, \|f_N\|_\infty B_N),
\]
where $B_i$ is a geodesic ball around $e$ such that
$\lambda_n(B_i) = \frac{\|f_i\|_{L^1(\M^n)}}{\|f_i\|_\infty}$.
\end{proposition}

\begin{proof}
We can assume without loss of generality, that already $f_i = f_i^\REARR$, since 
taking rearrangements neither changes $L^1$ nor $L^\infty$ norms. Using 
(\ref{eq:polar}), we obtain
\begin{align*}
  I_\Psi(f_1,\ldots, f_N) &= \int_{\M^n}\ldots\int_{\M^n} \Psi(x_1,\ldots, x_N)
							 \prod_{i=1}^N f_i(x_i)\, dx_1\ldots dx_N, \\
						  &= \int_{\SP^{n-1}}\ldots\int_{\SP^{n-1}}
							 \int_0^{R^\M}\ldots\int_0^{R^\M}
							 \Psi(x(t_1, u_1), \ldots, x(t_N, u_N)) \\
						  &\qquad\times\prod_{i=1}^N f_i(x(t_i, u_i))
						     \SN^{n-1}t_i\,
							 dt_1\ldots dt_N\, du_1\ldots du_N
\end{align*}
By the radial symmetry of $f_i$, we can write
$\tilde{f}_i(t_i) = f_i(x(t_i, u_i))$ to arrive at
\begin{align*}
  I_\Psi(f_1,\ldots, f_N) = \int_0^{R^\M}\ldots\int_0^{R^\M} \phi(t_1,
  \ldots, t_N) \prod_{i=1}^N \tilde{f}_i(t_i)\SN^{n-1}t_i\, dt_1\ldots
  dt_N.
\end{align*}
Now, applying Lemma \ref{lem:bathtub} successively to each coordinate and
noticing that
\[
  \lambda_n(B_i) = \int_0^{A_i} \SN^{n-1} t\, dt 
  				 = \frac{1}{\|f_i\|_\infty}
  				   \int_0^{R^\M}\tilde{f}_i(t)\SN^{n-1}t\, dt
  				 = \frac{\|f_i\|_{L^1(\M^n)}}{\|f_i\|_\infty},
\]
where the $A_i\in\R^+$ come from the lemma, yields the statement.
\end{proof}

\section{Proof of main theorem}

We split the proof of Theorem \ref{1:main-thm} into two parts: first, we show
how to pass from given functions to their symmetric decreasing rearrangements.
In a second step, we further move from radially symmetric, decreasing
functions to (multiples) of indicators of geodesic balls. For positive,
bounded, and integrable functions $f_1,\ldots, f_N$ on $\M^n$ write
\[
  I(f_1,\ldots, f_N) = \int_{\M^n}\ldots\int_{\M^n}
  					   U_1(\CONV\{x_1,\ldots, x_N\})
					   \prod_{i=1}^N f_i(x_i)\, dx_1\ldots dx_N.
\]
\begin{proposition}\label{4:propfRearr}
Let $f_1,\ldots, f_N\colon\M^n\to\R^+$ bounded, integrable.	Then
\[
  I(f_1\ldots, f_N) \geq I(f_1^\REARR\ldots, f_N^\REARR).
\]
\end{proposition}
\begin{proof}
For bounded, measurable subsets $K_1, \ldots, K_N\subseteq\M^n$ we set
$I(K_1, \ldots, K_N) := I(\IND_{K_1}, \ldots, \IND_{K_N})$. Our first
step is to show that $I(K_1, \ldots, K_N)\geq I(TK_1, \ldots, TK_N)$
for every two-point symmetrization $(H, \rho, T)$. To this end, for
$M\in {\mathcal M}_{n-1}^n$, let
\[
  I(K_1, \ldots, K_N; M) := \int_{K_1}\ldots\int_{K_N}
  							\chi(\CONV\{x_1,\ldots, x_N\}\cap M)\,
  							dx_1\ldots, dx_N.
\]
We want to investigate how the quantity 
$I(K_1, \ldots, K_N; M) + I(K_1, \ldots, K_N; \rho M)$ changes, when 
the $K_i$ are replaced by $TK_i$. Note that  
we have 
\[
  I(K_1, \ldots, K_N; \rho M) = I(\rho K_1, \ldots, \rho K_N; M)
\]
by the $\rho$-invariance of $\chi$. We begin by decomposing each $K_i$
according to the symmetrization
\[
  K_i = \underbrace{(K_i \cap \rho K_i)}_{K_i^\textrm{sym}} \dot{\cup} 
  			\underbrace{\left[(K_i \cap H^+)
  				\setminus K_i^\textrm{sym}\right]}_{K_i^\textrm{fix}} \dot{\cup} 
  			\underbrace{\left[(K_i \cap H^-)
  				\setminus K_i^\textrm{sym}\right]}_{K_i^\textrm{mov}},
\]
that is, $TK_i = K_i^\textrm{sym} \;\dot{\cup}\; K_i^\textrm{fix} \;\dot{\cup}\; \rho K_i^\textrm{mov}$.
Now, let 
$(x_1,\ldots, x_N)\in K_1\times\cdots\times K_N$ and introduce the
following labeling:
\begin{align*}
	\{a_1,\ldots, a_{N_0}\} 
		&:= \{x_i\ST x_i\in K_i^\textrm{sym}, 1\leq i\leq N\} \\
	\{b_1,\ldots, b_{N_1}\} 
		&:= \{x_i\ST x_i\in K_i^\textrm{fix}, 1\leq i\leq N\} \\
	\{c_1,\ldots, c_{N_2}\} 
		&:= \{x_i\ST x_i\in K_i^\textrm{mov}, 1\leq i\leq N\},
\end{align*}
where $N_0 + N_1 + N_2 = N$. For brevity we will use the notation
$\bar{x} := \rho x$ for $x\in\M^n$ and consider the tuples
\begin{align*}
	D_1 := (a_1,\ldots, a_{N_0}, 
			b_1,\ldots, b_{N_1},
			c_1,\ldots, c_{N_2}) 
			&\in \times_{i=1}^N \, K_i, \\
	D_2 := (\bar{a}_1,\ldots, \bar{a}_{N_0},
			b_1,\ldots, b_{N_1},
			c_1,\ldots, c_{N_2}) 
			&\in \times_{i=1}^N \, K_i, \\
	D_3 := (\bar{a}_1,\ldots, \bar{a}_{N_0},
			\bar{b}_1,\ldots, \bar{b}_{N_1},
			\bar{c}_1,\ldots, \bar{c}_{N_2}) 
			&\in \times_{i=1}^N \, \rho K_i, \\
	D_4 := (a_1,\ldots, a_{N_0},
			\bar{b}_1,\ldots, \bar{b}_{N_1},
			\bar{c}_1,\ldots, \bar{c}_{N_2}) 
			&\in \times_{i=1}^N \, \rho K_i,
\end{align*}
and
\begin{align*}
	E_1 := (a_1,\ldots, a_{N_0}, 
			b_1,\ldots, b_{N_1},
			\bar{c}_1,\ldots, \bar{c}_{N_2}) 
			&\in \times_{i=1}^N \, TK_i, \\
	E_2 := (\bar{a}_1,\ldots, \bar{a}_{N_0},
			b_1,\ldots, b_{N_1},
			\bar{c}_1,\ldots, \bar{c}_{N_2}) 
			&\in \times_{i=1}^N \, TK_i, \\
	E_3 := (\bar{a}_1,\ldots, \bar{a}_{N_0},
			\bar{b}_1,\ldots, \bar{b}_{N_1},
			c_1,\ldots, c_{N_2}) 
			&\in \times_{i=1}^N \, \rho TK_i, \\
	E_4 := (a_1,\ldots, a_{N_0},
			\bar{b}_1,\ldots, \bar{b}_{N_1},
			c_1,\ldots, c_{N_2}) 
			&\in \times_{i=1}^N \, \rho TK_i.
\end{align*}
Note that exchanging $c_l$ with $\bar{c}_l$, $1\leq l\leq N_2$, yields 
the mapping
\begin{equation}\label{4:eqCflip}
  D_1 \mapsto E_1, \quad D_2 \mapsto E_2, \quad 
  D_3 \mapsto E_3, \quad D_4 \mapsto E_4,
\end{equation}
whereas exchanging $b_k$ with $\bar{b}_k$, $1\leq k\leq N_1$ induces
\begin{equation}\label{4:eqBflip}
	D_1 \mapsto E_4, \quad D_2 \mapsto E_3, \quad 
	D_3 \mapsto E_2, \quad D_4 \mapsto E_1.
\end{equation}
We claim that
\begin{equation}\label{4:eqPointWise}
  \sum_{i=1}^4\chi(\CONV\{D_i\}\cap M) \geq   \sum_{i=1}^4\chi(\CONV\{E_i\}\cap M)
\end{equation}
for almost all $M\in\mathcal{M}_{n-1}^n$, that is, we tacitly assume
that $x_1,\ldots, x_N$ do not lie on $M$. We will verify the claim by
checking all possible positions of the points $a_j$, $b_k$, $c_l$
relative to $M$. In doing so, we mean by a \emph{pair} of points a
point and its reflection about $H$, that is $x$ and $\bar{x} = \rho
x$. We call a pair of points $x$ and $\bar{x}$ \emph{split} if they
lie on opposite sides of $M$.
\begin{itemize}
	\item \underline{Case 1}: \emph{None of the pairs of $b$'s are
          split.} By (\ref{4:eqBflip}), the terms on both sides of
          (\ref{4:eqPointWise}) are just a permutation of each other,
          thus there is equality in (\ref{4:eqPointWise}).
	
	\item \underline{Case 2}:
	\emph{None of the pairs of $c$'s are split.} By (\ref{4:eqCflip})
	and the same argument as in the first case, we have equality in 
	(\ref{4:eqPointWise}).
	
	\item \underline{Case 3}:
	\emph{There exist split pairs of $b$'s and split pairs of $c$'s.}
	Suppose that $\{b_k, \bar{b}_k\}$, $1\leq k\leq N_1$ and 
	$\{c_l, \bar{c}_l\}$, $1\leq l\leq N_2$ are split. Since
	$b_k, \bar{c}_l\in H^+$ and $\bar{b}_k, c_l\in H^-$, the geodesic
	segments $[b_k, c_l]$ and $[\bar{b}_k, \bar{c}_l]$ intersect in $H$.
	As $M$ divides $\M^n$ into two connected components, $b_k$ and
	$\bar{c}_l$ must lie on one side of $M$, whereas $\bar{b}_k$ and
	$c_l$ must lie on the other. Thus, the left hand side of 
	(\ref{4:eqPointWise}) equals $4$ and the inequality holds.
	\begin{figure}
		\begin{center}
		\resizebox{0.98\linewidth}{!} {
			\begin{tikzpicture}[scale=1.5, xscale=1]	

				\definecolor{darkgreen}{RGB}{50,110,0}

				\draw[dotted] (-2, 0) -- (2, 0) node[right]{$H$};
				\node[above] at (2, 0.8) {$H^+$};
				\node[below] at (2, -0.8) {$H^-$};
	
				\newcommand\Body{
						plot [smooth cycle, tension=1, rotate=40, scale=1.5]
						coordinates {(1, 0) (0.2, 0.5) (-1, 0) (0.2, -0.5)}
				}
	
	
	
	
				\coordinate (A) at (-0.4, -0.25);
				\coordinate (rA) at (-0.4, 0.25);
				\coordinate (B) at (0.8, 0.6);
				\coordinate (rB) at (0.8, -0.6);
				\coordinate (C) at (-0.9, -0.4);
				\coordinate (rC) at (-0.9, 0.4);
	
				\node at (B) [circle, fill, inner sep=1pt]{};
				\node at (B) [above right]{$b_k$};
				\node at (C) [circle, fill, inner sep=1pt]{};
				\node at (C) [below left]{$c_l$};
				\node at (rB) [circle, fill, inner sep=1pt]{};
				\node at (rB) [below right]{$\bar{b}_k$};
				\node at (rC) [circle, fill, inner sep=1pt]{};
				\node at (rC) [above left]{$\bar{c}_l$};
				\node at (A) [circle, fill, inner sep=1pt]{};
				\node at (A) [below right]{$a_j$};
				\node at (rA) [circle, fill, inner sep=1pt]{};
				\node at (rA) [above right]{$\bar{a}_j$};
				
				\fill[color=green, opacity=0.2] (A) -- (B) -- (C);
				\draw[thick, color=darkgreen] (A) -- (B) -- (C) -- cycle;
				\fill[color=green, opacity=0.2] (rA) -- (rB) -- (rC);
				\draw[thick, color=darkgreen] (rA) -- (rB) -- (rC) -- cycle;
				\fill[color=green, opacity=0.2] (rA) -- (B) -- (C);
				\draw[thick, color=darkgreen] (rA) -- (B) -- (C) -- cycle;
				\fill[color=green, opacity=0.2] (A) -- (rB) -- (rC);
				\draw[thick, color=darkgreen] (A) -- (rB) -- (rC) -- cycle;
	
				\draw[dashed]
				(-2, -0.2) node[below right] {$M$} -- (2, 0.5);
	    \end{tikzpicture}
	    \hspace{5mm}
	    \begin{tikzpicture}[scale=1.5, xscale=1]	
	  	  
	 	   \definecolor{darkgreen}{RGB}{50,110,0}
	    
	 	   \draw[dotted] (-2, 0) -- (2, 0) node[right]{$H$};
	 	   \node[above] at (2, 0.8) {$H^+$};
	     \node[below] at (2, -0.8) {$H^-$};
	    
	  	  \newcommand\Body{
	    		plot [smooth cycle, tension=1, rotate=40, scale=1.5]
	    		coordinates {(1, 0) (0.2, 0.5) (-1, 0) (0.2, -0.5)}
	    	}
	    
	    
	    
	    
	    	\coordinate (A) at (-0.4, -0.25);
		    \coordinate (rA) at (-0.4, 0.25);
		    \coordinate (B) at (0.8, 0.6);
		    \coordinate (rB) at (0.8, -0.6);
		    \coordinate (C) at (-0.9, -0.4);
		    \coordinate (rC) at (-0.9, 0.4);
	    
		    \node at (B) [circle, fill, inner sep=1pt]{};
		    \node at (B) [above right]{$b_k$};
		    \node at (C) [circle, fill, inner sep=1pt]{};
		    \node at (C) [below left]{$c_l$};
		    \node at (rB) [circle, fill, inner sep=1pt]{};
		    \node at (rB) [below right]{$\bar{b}_k$};
		    \node at (rC) [circle, fill, inner sep=1pt]{};
		    \node at (rC) [above left]{$\bar{c}_l$};
		    \node at (A) [circle, fill, inner sep=1pt]{};
	  	  \node at (A) [below right]{$a_j$};
		    \node at (rA) [circle, fill, inner sep=1pt]{};
		    \node at (rA) [above right]{$\bar{a}_j$};
	    
	    	\fill[color=green, opacity=0.2] (A) -- (B) -- (rC);
	    	\draw[thick, color=darkgreen] (A) -- (B) -- (rC) -- cycle;
	    	\fill[color=green, opacity=0.2] (rA) -- (rB) -- (C);
	    	\draw[thick, color=darkgreen] (rA) -- (rB) -- (C) -- cycle;
	    	\fill[color=green, opacity=0.2] (rA) -- (B) -- (rC);
	    	\draw[thick, color=darkgreen] (rA) -- (B) -- (rC) -- cycle;
	    	\fill[color=green, opacity=0.2] (A) -- (rB) -- (C);
	    	\draw[thick, color=darkgreen] (A) -- (rB) -- (C) -- cycle;
	    
	    	\draw[dashed]
	    	(-2, -0.2) node[below right] {$M$} -- (2, 0.5);
	    \end{tikzpicture}
		}
		\end{center}
		\caption{$\CONV\{D_i\}$ (left) and $\CONV\{E_i\}$ (right) for $i\in\{1,2,3,4\}$.}
	\end{figure}
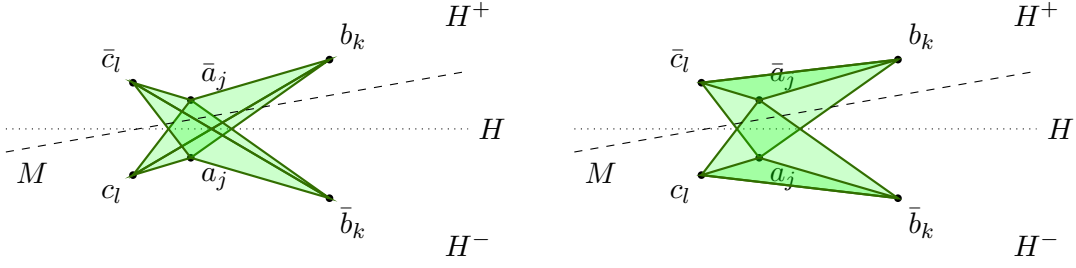
\end{itemize}
Integrating the pointwise inequality \eqref{4:eqPointWise} over
$K_1\times\cdots\times K_N$ yields
\begin{align*}
	2I(K_1, \ldots, &K_N; M) + 2I(K_1, \ldots, K_N; \rho M) \\
	&=
  2I(K_1, \ldots, K_N; M) + 2I(\rho K_1, \ldots, \rho K_N; M) \\
  &\geq
  2I(TK_1, \ldots, TK_N; M) + 2I(\rho TK_1, \ldots, \rho TK_N; M) \\
  &=
  2I(TK_1, \ldots, TK_N; M) + 2I(TK_1, \ldots, TK_N; \rho M),
\end{align*}
that is, the quantity $I(K_1, \ldots, K_N; M) + I(K_1, \ldots, K_N; \rho M)$
decreases whenever the sets $K_1, \ldots, K_N$ are replaced by $TK_1, \ldots, TK_N$.

Our next step is to use the layer-cake formula to generalize the
previous inequality to functions. Let $f_1,\ldots, f_N\colon\M^n\to\R^+$
be bounded integrable functions and set
\[
	I(f_1, \ldots, f_N; M) := \int_{\M^n}\ldots\int_{\M^n}
	\chi(\CONV\{x_1,\ldots, x_N\}\cap M)\prod_{i=1}^N f_i(x_i)\,
	dx_1\ldots, dx_N.
\]
Indeed, we have that
\begin{align*}
	I &(f_1,\ldots, f_N; M) + I(f_1,\ldots, f_N; \rho M)= \\
	& =	
	\int_0^\infty \ldots \int_0^\infty
		I(\{f_1>t_1\}, \ldots, \{f_N>t_N\}; M) \\
		& \hspace{3cm}+ I(\{f_1>t_1\}, \ldots, \{f_N>t_N\}; \rho M)\,
		dt_1\ldots dt_N \\	
	& \geq
	\int_0^\infty \ldots \int_0^\infty
		I(T\{f_1>t_1\}, \ldots, T\{f_N>t_N\}; M) \\
		& \hspace{3cm}+ I(T\{f_1>t_1\}, \ldots, T\{f_N>t_N\}; \rho M)\,
		dt_1\ldots dt_N \\
	& =
	\int_0^\infty \ldots \int_0^\infty
		I(\{Tf_1>t_1\}, \ldots, \{Tf_N>t_N\}; M) \\
		& \hspace{3cm}+ I(\{Tf_1>t_1\}, \ldots, \{Tf_N>t_N\}; \rho M)\,
		dt_1\ldots dt_N \\
	&= I(Tf_1,\ldots, Tf_N; M) + I(Tf_1,\ldots, Tf_N; \rho M).
\end{align*}
Here, we used the layer-cake representation 
$f(x) = \int_0^\infty \mathbbm{1}_{\{f>t\}}(x)\, dt$, identity \eqref{2:eqTPlevel}, 
and the above inequality for sets. We can now apply Proposition \ref{3:propTPRadial}
to the bounded function 
$\Psi(x_1,\ldots, x_N) = \chi(\CONV\{x_1,\ldots, x_N\}\cap M) 
+ \chi(\CONV\{x_1,\ldots, x_N\}\cap \rho M)$
to obtain
\[
  I(f_1,\ldots, f_N; M) + I(f_1,\ldots, f_N; \rho M) \geq
  I(f_1^\REARR,\ldots, f_N^\REARR; M) + I(f_1^\REARR,\ldots, f_N^\REARR; \rho M).
\]
The proof of the inequality is now completed by integrating $M$ over $\mathcal{M}_{n-1}^n$.
%
\end{proof}
	Note, that since Proposition \ref{4:propfRearr} does not require the functions $f_i$ 
	to have proper support in the spherical case, Corollaries \ref{1:main-cor} 
	and \ref{1:cor:randBS} also hold for non-proper sets $K\subseteq\SP^n$.

\begin{proposition}
Let $N\in\N$, $f_1,\ldots, f_N\colon\M^n\to\R^+$ bounded,
integrable, and with proper support, if $\M^n = \SP^n$.	Then
\[
  I(f_1^\REARR,\ldots, f_N^\REARR) \geq I(\|f_1\|_\infty B_1, \ldots, \|f_N\|_\infty B_N)
\]
where $B_i$ is a geodesic ball around $e$ such that
$\lambda_n(B_i) = \frac{\|f_i\|_{L^1(\M^n)}}{\|f_i\|_\infty}$.
\end{proposition}
\begin{proof}
We use polar coordinates around $e\in\M^n$ (see Section \ref{sec:background}),
\[
  x(t, u) = e\CS t + u \SN t, \qquad t\in\left[0, R^\M\right], u\in\SP^{n-1},
\]
and appeal to Proposition \ref{3:propRadCap}. In doing so, we will
justify monotonicity in each coordinate of the following function:
\begin{align*}
  \phi(t_1,\ldots, t_N) = \int_{\SP^{n-1}}\ldots \int_{\SP^{n-1}}
  &\chi\left(\CONV\{(x(t_1, u_1), \ldots, x(t_N, u_N)\} \cap M\right) \\
  +\, &\chi\left(\CONV\{(x(t_1, u_1), \ldots, x(t_N, u_N)\} \cap M^e\right)
  \, du_1\ldots du_N,
\end{align*}
where $M\in\mathcal{M}^n_{n-1}$ is fixed and $x^e := -x + (x\cdot e)e$ denotes 
the geodesic reflection of $x\in\M^n$ about $e$, that is, orthogonal reflection
about $\SPAN\{e\}$ in $\R^{n+1}$.

Without loss of generality, we show that $\phi$ is increasing in $t_1 =: t$.
We fix $t_2,\ldots, t_N$ and $u_1,\ldots, u_N$ and write $x(t) := x(t, u_1)$ and
$x_i := x(t_i, u_i)$, $2\leq i\leq N$. Define
\begin{align*}
  \alpha_1(t) &:= \chi(\CONV\{x(t), x_2,\ldots, x_N\}\cap M), &
  \alpha_2(t) &:= \chi(\CONV\{x(t)^e, x_2,\ldots, x_N\}\cap M), \\
	\alpha_3(t) &:= \chi(\CONV\{x(t), x_2^e,\ldots, x_N^e\}\cap M), &
	\alpha_4(t) &:= \chi(\CONV\{x(t)^e, x_2^e,\ldots, x_N^e\}\cap M),	
\end{align*}
and set $\alpha(t) := \alpha_1(t) + \alpha_2(t) + \alpha_3(t) + \alpha_4(t)$. 
Note that we have $\alpha_1 = \alpha_4^e$ and $\alpha_2 = \alpha_3^e$.
Our goal is to show that the function 
$\alpha\colon\left[0, R^\M \right]\to\{0,1,2,3,4\}$ is increasing in $t$. 
We denote by $X:=\CONV\{x_2,\ldots, x_N\}$ and consider the following cases:
\begin{itemize}
	\item\underline{Case 1}: $e\in X$, and thus, $e\in X^e$. For 
	$s\leq t$, we have $[e, x(s)] \subseteq [e, x(t)]$ as geodesic segments. 
	Therefore
	\[
	  \alpha_1(s) = \chi(\CONV\{[e, x(s)]\cup X\}\cap M) 
	    \leq \chi(\CONV\{[e,x(t)]\cup X\} \cap M)
	    = \alpha_1(t), 
	\]
	and similarly for $\alpha_2, \alpha_3, \alpha_4$, hence, 
	$\alpha(s)\leq\alpha(t)$.
	
	\item\underline{Case 2}: $e\notin X$, and thus, $e\notin X^e$.
	\begin{itemize}
		\item\underline{Case 2a}: \emph{M meets both $X$ and $X^e$}. 
		Here, $\alpha(t)\equiv 4$.
		
		\item\underline{Case 2b}: \emph{M meets $X$ but not $X^e$}. We first 
		show that in this case $e$ must lie on the same side of $M$ as $X^e$. 
		Assume the opposite, that is, $e$ lies opposite of $X^e$. Since $M$
		meets $X$, there exist points of $X$ on either side of $M$. Therefore, 
		we find $y\in X$ lying opposite of $e$. But then $y^e\in X^e$, and thus 
		the segment $[y, y^e]$ also lies opposite of $e$. This is a
		contradiction, as $e\in[y, y^e]$ (here we use the assumption
		that, if $\M^n = \SP^n$, the functions $f_1,\ldots, f_N$ have proper 
		support, and thus their rearrangements, $f^\REARR_1,\ldots, f^\REARR_N$ 
		are supported in $\INT\SP_e^+$).
		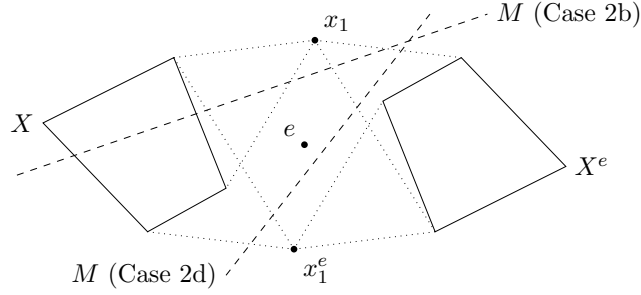
\begin{figure}
		\begin{center}
			\resizebox{90mm}{!} {
				\begin{tikzpicture}[scale=0.65, xscale=1.2]
				\path (-6, -3.5) -- (6, 3.5);
		
		
				\node at (0, 0) [circle, fill, inner sep=1pt]{};
				\node at (0, 0) [above left]{$e$};
		
				\draw (-1.5, -1) -- (-2.5, 2) -- (-5, 0.5) node[left] {$X$}-- (-3, -2) -- cycle;
				\draw (1.5, 1) -- (2.5, -2) -- (5, -0.5) node[right] {$X^e$}-- (3, 2) -- cycle;
		
				\coordinate (X1) at (0.2, 2.4);
				\node at (X1) [circle, fill, inner sep=1pt]{};
				\node at (X1) [above right]{$x_1$};
		
				\coordinate (X1e) at (-0.2, -2.4);
				\node at (X1e) [circle, fill, inner sep=1pt]{};
				\node at (X1e) [below right]{$x_1^e$};
		
				\draw[dotted] (-2.5, 2) -- (X1) -- (-1.5, -1);
				\draw[dotted] (2.5, -2) -- (X1e) -- (1.5, 1);
		
				\draw[dotted] (-2.5, 2) -- (X1e) -- (-3,-2);
				\draw[dotted] (2.5, -2) -- (X1) -- (3,2);
				
				\draw[dashed] (-5.5, -0.7) -- (3.5, 3) node[right] {$M$ (Case 2b)};
				\draw[dashed] (-1.5, -3) node[left] {$M$ (Case 2d)} -- (2.4, 3);		
				\end{tikzpicture}
			}
		\end{center}
		\caption{Different positions of the hypersurface $M$ in Cases 2b and 2d.}
		\end{figure}
		Hence, for $t$ small enough, we have
		\[
		\alpha_1(t) = 1, \quad \alpha_2(t) = 1, \quad
		\alpha_3(t) = 0, \quad \alpha_4(t) = 0,
		\]
		that is, $\alpha(t)=2$. As $t$ increases, as soon as either $x(t)$ or 
		$x(t)^e$ cross $M$, $\alpha_3(t) = 1$ or $\alpha_4(t) = 1$, that is, 
		$\alpha(t) = 3$. 
		
		\item\underline{Case 2c}: \emph{M meets $X^e$ but not $X$} is similar 
		to Case 2b.
		
		\item\underline{Case 2d}: \emph{M meets neither $X$ nor $X^e$}. If $X$
		and $X^e$ lie on opposite sides of $M$, then $\alpha(t)\equiv 2$ is 
		essentially constant, as $\alpha_1(t) + \alpha_3(t)\equiv 1$ and
		$\alpha_2(t) + \alpha_4(t)\equiv 1$ (except for at most one value of
		$t$, where $x_1$ or $x_1^e$ might lie on $M$).
		If $X$ and $X^e$ lie on the same side of $M$, then so does 
		$e$, and thus, $\alpha(t) = 0$ for small $t$, and $\alpha(t) = 2$, as 
		soon as $x(t)$ or $x(t)^e$ cross $M$, since then 
		$\alpha_1(t) = 1$, $\alpha(t)_3 = 1$ or 
		$\alpha(t)_2 = 1$, $\alpha(t)_4 = 1$, respectively.
	\end{itemize}
\end{itemize}
Since the above shows that $\alpha(t)$ is either increasing or constant a.e.\
for every choice of 
$u_1, \ldots, u_N$, integrating over $\SP^{n-1}\times\cdots\times\SP^{n-1}$ 
yields that 
\[
	2\phi(t, t_2, \ldots, t_N) =
	\int_{\SP^{n-1}}\ldots\int_{\SP^{n-1}}
	\alpha(t)\, du_1\ldots du_N     
\]
is increasing in $t$ as well. Hence, an application of Proposition 
\ref{3:propRadCap} gives
\begin{align*}
	I(f_1^\REARR, &\ldots, f_N^\REARR; M) + 
	I(f_1^\REARR,\ldots, f_N^\REARR; M^e) \\
	&\geq 
	I(\|f_1\|_\infty B_1, \ldots, \|f_N\|_\infty B_N; M) +
	I(\|f_1\|_\infty B_1, \ldots, \|f_N\|_\infty B_N; M^e)
\end{align*}
Once again, integrating $M$ over $\mathcal{M}_{n-1}^n$ concludes the proof.
\end{proof}

\vspace{0.6cm}

\noindent {{\bf Acknowledgments} T.\ Hack was
	supported by the European Research Council (ERC), Project number: 306445, and the Austrian Science Fund (FWF), Project numbers:
	Y603-N26 and P31448-N35. 
	P.\ Pivovarov was supported by the NSF grant DMS-1612936.

\begin{small}
	\[ \begin{array}{ll}
	\mbox{Thomas Hack} \hspace{130pt}      & \mbox{Peter Pivovarov} \\
	\mbox{Vienna University of Technology} & \mbox{University of Missouri} \\
	\mbox{thomas.hack@tuwien.ac.at}        & \mbox{pivovarovp@missouri.edu}
	\end{array}\]	
\end{small}
	
\end{document}